\documentclass{birkmult}

\usepackage{amsmath}
\usepackage{amssymb}
\usepackage{tensor}

\renewcommand{\subjclassname}{\textup{2000} Mathematics Subject Classification}

\newtheorem{theorem}{Theorem}[section]
\newtheorem{lemma}[theorem]{Lemma}

\theoremstyle{definition}

\newtheorem{example}[theorem]{Example}

\theoremstyle{remark}
\newtheorem{remark}[theorem]{Remark}

\numberwithin{equation}{section}

 \def\R{{\mathbb R}}
 \def\E{{\mathcal E}} 
 \def\eR{{\mathbb R} \cup \{+\infty\}}

 \def\N{{\mathbb N}}

 \def\cH{{\mathcal H}}

 \def\mr{{\rm MR}_p (X, D_A, D_B)}

 \def\mr1{{\rm MR}_{p,1} (X, D_A, D_B)}

\DeclareMathOperator{\Dom}{dom}

\newcommand{\dom}[1]{\Dom{#1}}

\begin{document}
%-------------------------------------------------------------------------
%editorial commands: to be inserted by the editorial office
%
%\firstpage{1}
%\volume{228}
%\DOI{003-0001}
%\seriesextraline{This is the Concrete Title of this Book\br H.E. R and S.T.C. W, Eds.}
%\firstpage{1}

\title[Nonsmooth infinite-dimensional gradient systems]{The Kurdyka-{\L}ojasiewicz-Simon inequality and stabilisation in nonsmooth infinite-dimensional gradient systems}

\author{Ralph Chill}
\address{Technische Universit\"at Dresden,
Institut f\"ur Analysis,
01062 Dresden,
Germany}
\email{ralph.chill@tu-dresden.de}
%\thanks{}

\author{Sebastian Mildner}

\address{Technische Universit\"at Dresden,
Institut f\"ur Analysis,
01062 Dresden,
Germany}
\email{sebastian.mildner@tu-dresden.de}

% General info
%\subjclass{}

\date{\today}

\renewcommand{\subjclassname}{\textup{2000} Mathematics Subject Classification}


   %analytic semigroup, Banach, 
   %algebras, non-quasinilpotent, quasi\-nil\-potent}        % the keywords

% \MRSubClass{47D03, 46J40, 46H20} 

%\subjclass{Primary 99Z99; Secondary 00A00}

%\keywords{Class file, journal}

%\date{\today}

%\dedicatory{}
%%% ----------------------------------------------------------------------

\begin{abstract}
We state and prove a stabilisation result for solutions of abstract gradient systems associated with nonsmooth energy functions on infinite dimensional Hilbert spaces. One feature is that in this general setting the assumption on the range of the solution can be considerably relaxed, which considerably simplifies the applicability of the stabilisation result even in the case of smooth energies. 
\end{abstract}
 
\renewcommand{\subjclassname}{\textup{2000} Mathematics Subject Classification}

%%% ----------------------------------------------------------------------
\maketitle
%%% ----------------------------------------------------------------------

\section{Introduction}

The {\L}ojasiewicz gradient inequality for real analytic functions on $\R^N$ \cite{Lo63,Lo65} and its generalisations to functions definable in $o$-minimal structures \cite{Ku98} or to smooth functions on infinite dimensional Hilbert spaces \cite{Si83,Je98b,Ch03} have proved to be major tools in the study of asymptotic behaviour of gradient and gradient-like systems. The {\L}ojasiewicz-Simon inequality for smooth energy functions on infinite-dimensional Hilbert spaces has been applied in order to prove stabilisation of bounded solutions of many parabolic equations such as diffusion equations, Cahn-Hilliard type equations for describing phase separation phenomena, or geometric evolution equations, but also to hyperbolic equations such as damped wave equations; the literature being vast, we merely refer to the monographs by Haraux \& Jendoubi \cite{HaJe15}, Huang \cite{Hu06} and the references therein. 

In this article, we consider nonsmooth gradient systems in infinite dimensional Hilbert spaces, associated with semiconvex, lower semicontinuous energy functions and their subgradients. We show that the Kurdyka-{\L}ojasiewicz-Simon inequality may also applied in this general setting in order to prove stabilisation of solutions of gradient systems. The point of this article is, however, not only this generalisation. Unlike in the situation of smooth energy functions, which are at least continuously differentiable functions defined on (open subsets of) a Banach space, a natural energy space is not present in the case of energy functions defined on a Hilbert space and taking values in the extended real line. The role of energy space is taken over by the effective domain which, however, carries in general no linear structure. 

This article starts with a small but useful observation. The effective domain of a function $\E$ on a metric space $M$ always carries a natural topology $\tau_\E$, so that $(\dom{\E} , \tau_\E )$ is continuously embedded into $(M,d)$ and so that $\E$ is continuous on $(\dom{\E} , \tau_\E )$. Actually, we take the coarsest topology with these two properties. We show that in the case of the classical Dirichlet energy of the Neumann-Laplacian on $L^2 (\Omega )$, but also for semilinear perturbations of this energy, this natural topology coincides with the norm topology on the Sobolev space $H^1 (\Omega )$. 

This small observation is used in the second part where we state and prove the stabilisation result for global, bounded solutions of associated gradient systems. This result uses the Kurdyka-{\L}ojasiewicz-Simon inequality, named after the Kurdyka-{\L}ojasiewicz inequality for functions definable in $o$-minimal structures and after the {\L}ojasiewicz-Simon inequality for functions defined on Hilbert spaces. A new feature is that the usual assumption of relative compactness of the global solution in the energy space (or in the effective domain equipped with the topology mentioned above) can be considerably weakened to the assumption of relative compactness of the solution in the ambient Hilbert space. In many applications of the {\L}ojasiewicz-Simon inequality the verification of the relative compactness of the range of the solution in the energy space required a lot of efforts and advanced techniques, while the relative compactness of the range of the solution in the ambient Hilbert space often follows from a standard application of Rellich-Kondrachov. Our result thus seems to be of interest even in the case of smooth energies with effective domains having a linear structure.

\section{Topology and metric induced by the energy}

Let $(M,d)$ be a metric space and let $\E : M \to\eR$ be an energy function with values in the extended real line. We suppose that $\E$ is {\em proper} in the sense that the {\em effective domain} $\dom{\E} := \{ \E <+\infty\}$ is nonempty.  We equip $\dom{\E}$ with a topology $\tau_\E$, namely the coarsest topology for which the natural embedding $\dom{\E}\to M$ and the mapping $\E : \dom{\E} \to \R$ are continuous. A net $(u_\alpha)$ in $\dom{\E}$ thus converges to $u\in\dom{\E}$ with respect to the topology $\tau_\E$ if and only if $\lim_{\alpha} d(u_\alpha ,u) = 0$ and $\lim_{\alpha} \E (u_\alpha) = \E (u)$. As a consequence of the simple structure of the topology $\tau_\E$, we have the following lemma. 

\begin{lemma} \label{lem.0}
The topology $\tau_\E$ is metrizable. For example, the topology $\tau_\E$ is induced by the metric $d_\E : \dom{\E}\times \dom{\E}\to\R$ given by 
\[
 d_\E (u,v) := d(u,v) + |\E (u) - \E (v)| \quad (u,\, v\in \dom{\E}) .
\]
\end{lemma}

\begin{example} \label{ex.1}
On the Hilbert space $H = L^2 (\Omega )$ ($\Omega\subseteq\R^N$ open) we consider the function $\E_1 : L^2 (\Omega ) \to \eR$ given by $\E_1 (u) = \frac12 \int_\Omega |\nabla u|^2$ with effective domain $\dom{\E_1} = H^1 (\Omega )$. A sequence $(u_n)$ in $H^1 (\Omega )$ converges with respect to $\tau_{\E_1}$ to some element $u\in H^1 (\Omega )$ if and only if $\lim_n u_n = u$ in $L^2 (\Omega )$ and $\lim_n \E_1 (u_n ) = \E_1 (u)$ in $\R$. As a consequence, if a sequence $(u_n)$ converges to $u\in H^1 (\Omega )$ with respect to $\tau_{\E_1}$, then necessarily $(u_n)$ is bounded in $(H^1 (\Omega ) , \|\cdot \|_{H^1} )$. By reflexivity of $H^1 (\Omega )$ and by continuity of the embedding $H^1 (\Omega )\hookrightarrow L^2 (\Omega )$, the sequence $(u_n )$ thus converges weakly to $u\in H^1 (\Omega )$. However, the convergence in $\tau_{\E_1}$ implies in addition that $\lim_n \| u_n\|_{H^1} = \| u\|_{H^1}$, and hence $(u_n)$ converges to $u$ in the norm topology of $H^1 (\Omega )$. Obviously, the converse implication -- saying that convergence in the norm topology implies convergence in $\tau_{\E_1}$ -- is true, too, and hence, using also Lemma \ref{lem.0}, both topologies coincide. 
\end{example}

\begin{lemma} \label{lem.1}
Let $(M,d)$ be a metric space. Let $\E_1$, $\E_2 :M\to\eR$ be two functions, and let $\E := \E_1 + \E_2$. Then:
\begin{itemize}
 \item[(a)] If $\E_2$ is continuous with respect to the topology $\tau_{\E_1}$, then $\tau_\E$ is coarser than $\tau_{\E_1}$.
 \item[(b)] If $\E_2$ is continuous with respect to the topology in $M$, then $\tau_\E = \tau_{\E_1}$. 
\end{itemize}
\end{lemma}

\begin{proof}
(a) By assumption and by definition of $\tau_{\E_1}$, both $\E_1$ and $\E_2$ are continuous with respect to the topology $\tau_{\E_1}$, and hence $\E$ is continuous with respect to this topology. By definition again, the topology $\tau_\E$ must be coarser than the topology $\tau_{\E_1}$. 

(b) This follows by symmetry ($\E_1 = \E - \E_2$) and by applying (a).  
\end{proof}

\begin{example} \label{ex.2}
On the Hilbert space $H = L^2 (\Omega )$ we consider the function $\E$ given by
\begin{align*}
 \E (u) & = \frac12 \int_\Omega |\nabla u|^2 + \int_\Omega F(u) \\
 & = \E_1 (u) + \E_2 (u) ,
\end{align*}
where $\E_1$ is as in Example \ref{ex.1} and $\E_2 (u) = \int_\Omega F(u)$ for some function $F\in C^1 (\R )$ with globally Lipschitz continuous derivative $F'$. The function $\E_2$ is continuous with respect to the norm topology in $L^2 (\Omega )$. By Example \ref{ex.1} and Lemma \ref{lem.1}, $\tau_\E$ coincides with the norm topology in $H^1 (\Omega )$. 
\end{example}

\begin{example} \label{ex.3}
More generally, if $\E$ is a function on a Hilbert space $H$, if the effective domain $V:= \dom{\E}$ is a subspace of $H$, equipped with a seminorm $|\cdot |_V$ such that $\|\cdot \|_V := |\cdot |_V + \|\cdot \|_H$ is a complete norm and $(V,\|\cdot \|_V )$ is a dual Banach space, and if $\E$ is a function of this seminorm (that is, $\E = f\circ |\cdot |_V$ for some continuous $f:\R\to\R$), then $\tau_\E$ is in general coarser than the norm topology of $V$. For example, consider the choice $H= L^2 (\Omega )$, $V = L^2 (\Omega ) \cap BV (\Omega )$ and $\E (u) = |u|_{TV}$ (the total variation seminorm). 
\end{example}

\begin{example} \label{ex.4}
Let $\E : M\to\eR$ be a function on a metric space $(M,d)$. Given a subset $C\subseteq M$, we define the characteristic function $1_C : M \to\eR$ by
\[
 1_C (u) := \begin{cases}
             0 & \text{if } u\in C , \\[2mm]
             +\infty & \text{else,} 
            \end{cases}
\]
and we let $\E_C := \E + 1_C$. Then $\E_C$ is proper if $\dom{\E_C} = \dom{\E}\cap C \not= \emptyset$. The topology $\tau_{\E_C}$ is the topology induced by $\tau_\E$ on $\dom{\E_C}$. Indeed, the metrics $d_\E$ and $d_{\E_C}$ (compare with Lemma \ref{lem.0}) coincide on $\dom{\E_C}$. 
\end{example}

\section{Stabilisation of global solutions of nonsmooth gradient systems}

Let $H$ be a Hilbert space and let $\E : H\to\eR$. We say that $\E$ is {\em semiconvex}, if there exists $\omega\in\R$ such that $u\mapsto \E (u) + \frac{\omega}{2} \, \| u\|_H^2$ is convex. The {\em subgradient} of $\E$ is the relation
\begin{align*}
 \partial\E & := \{ (u,f)\in H\times H : u\in\dom{\E} \text{ and for every } v\in H \\
 & \phantom{\{ (u,f)\in H\times H : } \liminf_{\lambda\to 0+} \frac{\E (u+\lambda v) - \E (u)}{\lambda} \geq \langle f , v\rangle_H \} . 
\end{align*}
For semiconvex $\E$ and $\omega\in\R$ large enough,
\begin{align*}
 \partial\E & = \{ (u,f)\in H\times H : u\in\dom{\E} \text{ and for every } v\in H \\
 & \phantom{\{ (u,f)\in H\times H : } \E (v) - \E (u) + \frac{\omega}{2} \, \| v-u\|_H^2 \geq \langle f , v-u\rangle_H \} . 
\end{align*}
For every $u\in H$ we set $\partial\E (u) := \{ f\in H : (u,f)\in\partial\E \}$, which is a closed and convex set. Furthermore, we define the {\em slope} $|\partial\E (u)| := \inf \{ \| f\|_H : f\in\partial\E (u)\}$, with the convention $\inf \emptyset = \infty$. If $\partial\E (u)$ is nonempty, then $|\partial\E (u)| = \| P_{\partial\E (u)} 0\|_H$, where $P_{\partial\E (u)}$ denotes the orthogonal projection onto $\partial\E (u)$. 

\begin{lemma} \label{lem.subgradient.closed}
 Let $\E : H\to\eR$ be proper, semiconvex and lower semicontinuous. Let $((u_n,f_n))$ be a sequence in $\partial\E$ and $(u,f)\in H\times H$ such that 
\[
 \lim_{n\to\infty} u_n = u \text{  and  } {\rm weak-\!\!\!}\lim_{n\to\infty} f_n = f .
\]
Then
\[
 (u,f)\in\partial\E \text{  and  } \lim_{n\to\infty} \E (u_n) = \E (u) .
\]
\end{lemma}

\begin{proof}
By the characterisation of the subgradient of semiconvex functions, for some $\omega\in\R$ large enough, and for every $v\in H$ and every $n\in\N$,
\begin{equation} \label{eq.energy1}
 \E (v) \geq \E (u_n) + \langle f_n , v-u_n \rangle - \frac{\omega}{2} \, \| v-u_n\|_H^2 .
\end{equation} 
By taking the limit inferior on the right-hand side of this inequality, as $n\to\infty$, and by using the lower semicontinuity of $\E$, 
\[
 \E (v) \geq \E (u) + \langle f,v-u\rangle - \frac{\omega}{2} \, \| v-u\|_H^2 \text{ for every } v\in H .
\]
This inequality implies first (choose $v\in\dom{\E}$!) that $u\in\dom{\E}$, and second that $(u,f)\in\partial\E$. Choosing now $v=u$ in \eqref{eq.energy1}, and taking the limit superior on the right-hand side of that inequality, one obtains $\E (u) = \lim_{n\to\infty} \E (u_n )$. 
\end{proof}

If $\E$ is a proper, semiconvex, lower semicontinuous function on $H$, then the gradient system 
\begin{equation} \label{gs}
 \dot u + \partial\E (u) \ni f 
\end{equation}
admits for every $u_0 \in \dom{\E}$ and every $f\in L^2 (\R_+ ;H)$ a unique strong solution $u\in H^1_{loc}(\R_+ ;H)$ satisfying the initial condition $u(0) = u_0$ \cite[Th\'eor\`eme 3.6, p.72]{Br73}, \cite[Theorem 4.11]{Bar10}. Strong solution means that $(u(t) , f(t) - \dot u (t)) \in\partial\E$ for almost every $t\in\R_+$. For every strong solution $u$ the composition $\E (u)$ is absolutely continuous, and for almost every $t\in\R_+$ the energy equality
\begin{equation} \label{eq.energy}
 \frac{d}{dt} \E (u) = - \frac12 \| \dot u \|_H^2 - \frac12 \| P_{\partial\E (u)}f \|_H^2 + \frac12 \| f\|_H^2 
\end{equation}
holds (use \cite[Lemma 4.4]{Bar10} or compare with \cite[Th\'eor\`eme 3.6, p.72]{Br73}, \cite[Theorem 2.3.3]{AmGiSa05}). In particular, the function $\cH :\R_+ \to\R_+$, defined by 
\begin{equation} \label{eq.h}
 \cH (t) = \E (u(t)) + \frac12 \, \int_t^\infty \| f(s)\|_H^2 \; ds ,
\end{equation}
is absolutely continuous and decreasing. 

\begin{lemma} \label{lem.omega-limit}
 Let $\E : H\to\eR$ be proper, semiconvex and lower semicontinuous and $f\in L^2 (\R_+ ;H)$. Let $u\in H^1_{loc} (\R_+ ;H )$ be a solution of the gradient system \eqref{gs}, and consider its {\em $\omega$-limit set}
\[
 \omega (u) := \{ \varphi\in H : \exists (t_n)\nearrow\infty \text{ s.t. } \lim_{n\to\infty} u(t_n) = \varphi \text{ in } H \} .
\]
Then:
\begin{itemize}
 \item[(a)] For every $\varphi\in\omega (u)$ one has $\lim_{t\to\infty} \E (u(t)) = \E (\varphi )$.
 \item[(b)] The function $\E$ is constant on $\omega (u)$.
 \item[(c)] One has
\[
  \omega (u) = \{ \varphi\in H : \exists (t_n)\nearrow\infty \text{ s.t. } \lim_{n\to\infty} u(t_n) = \varphi \text{ w.r.t. } \tau_\E \} .
\]
\end{itemize}
\end{lemma}

\begin{proof}
Let $\varphi\in\omega (u)$, and let $(t_n)$ be a sequence in $\R_+$ such that $\lim_{n\to\infty} t_n = \infty$ and $\lim_{n\to\infty} u(t_n) = \varphi$ in $H$. Let $\cH$ be the function defined in \eqref{eq.h}. By lower semicontinuity of $\E$,
\[
 \liminf_{n\to\infty} \cH (t_n) = \liminf_{n\to\infty} \E (u(t_n)) \geq \E (\varphi ) ,
\]
so that $\cH$ is bounded from below. Since $\cH$ is also decreasing, 
\begin{equation} \label{eq.e}
 \lim_{t\to\infty} \cH (t) = \lim_{t\to\infty} \E (u(t)) \text{ exists.}
\end{equation}
Moreover, by the energy equality,  $\dot u \in L^2 (\R_+ ;H)$. From here we deduce, for every $s\in [0,1]$,
\begin{align*}
 \limsup_{n\to\infty} \| u(t_n +s ) - \varphi \|_H & \leq \limsup_{n\to\infty} (  \| u (t_n +s ) -u(t_n) \|_H + \| u(t_n)- \varphi \|_H ) \\
 & \leq \limsup_{n\to\infty} \left( \int_0^1 \| \dot u (t_n +r )\|_H \; dr + \| u(t_n)- \varphi \|_H \right) \\
 & = 0 . 
\end{align*}
Since $\dot u$, $f\in L^2 (\R_+ ;H)$ and $(u(t) , f(t) - \dot u (t))\in\partial\E$ for almost every $t$, we thus find a sequence $(s_n)\in [0,1]$ (depending on the representatives of the measurable functions $f$ and $\dot u$) such that $(u(t_n+s_n), f(t_n+s_n) - \dot u (t_n+s_n))\in\partial\E$,
\[
 \lim_{n\to\infty} u(t_n+s_n) = \varphi \text{ and } \lim_{n\to\infty} (f(t_n+s_n) - \dot u (t_n+s_n)) = 0 . 
\]
By Lemma \ref{lem.subgradient.closed}, this implies $(\varphi , 0)\in\partial\E$ and 
\[
 \lim_{n\to\infty} \E (u(t_n+s_n)) = \E (\varphi ) .
\]
From here and the convergence of $\E(u)$ (see \eqref{eq.e}) follows (a). Assertions (b) and (c) are direct consequences of (a). 
\end{proof}

We say that a function $\E : H \to\eR$ satisfies the {\em Kurdyka-{\L}ojasiewicz-Simon inequality} on a set $U\subseteq H$ if there exists a strictly increasing $\Theta\in W^{1,1}_{loc} (\R )$ such that $| \partial (\Theta\circ \E ) (v) | \geq 1$ for every $v\in U$ with $0\not\in\partial\E (v)$.

\begin{theorem} \label{thm.main}
Let $H$ be a Hilbert space and let $\E : H\to\eR$ be proper, semiconvex and lower semicontinuous. Let $u\in H^1_{loc} (\R_+;H)$ be a global strong solution of the gradient system \eqref{gs} with $f=0$. Assume that there exists $\varphi\in\omega (u)$ such that $\E$ satisfies the Kurdyka-{\L}ojasiewicz-Simon inequality in a $\tau_\E$-neighbourhood of $\varphi$. Then $u$ has finite length in $H$ and $\lim_{t\to\infty} u(t) = \varphi$ in $\tau_\E$.  
\end{theorem}

For the proof of Theorem \ref{thm.main}, we need the following chain rule.

\begin{lemma} \label{lem.chain-rule}
Let $\E : H\to\eR$ be proper, $u\in\dom{\E}$ and let $\Theta : \R\to\R$ be continuous, strictly increasing and differentiable at $\E(u)$. Then $\Theta' (\E (u)) \,\partial\E (u) \subseteq \partial (\Theta \circ \E ) (u)$. Moreover, if $\Theta' (\E (u)) \not= 0$, then $\Theta' (\E (u)) \,\partial\E (u) = \partial (\Theta \circ \E ) (u)$.
\end{lemma}

\begin{proof}
Let $f\in\partial\E (u)$ and $v\in H$. Let $\varepsilon >0$. Then there exists $\delta >0$ such that
\[
 \inf_{\lambda\in (0,\delta )} \frac{\E (u+\lambda v)-\E (u)}{\lambda} \geq \langle f, v\rangle - \varepsilon , 
\]
that is,
\[
 \E (u+\lambda v) \geq \E (u) + \lambda \, (\langle f,v\rangle - \varepsilon ) \text{ for every } \lambda \in (0,\delta ). 
\]
Due to the monotonicity of $\Theta$, we obtain
\begin{align*}
 \frac{(\Theta\circ \E ) (u+\lambda v) - (\Theta\circ\E ) (u)}{\lambda} & \geq \frac{\Theta (\E (u) +\lambda (\langle f,v\rangle -\varepsilon )) - \Theta (\E (u))}{\lambda} \\
 & \to \Theta' (\E (u)) \, (\langle f,v\rangle - \varepsilon) \quad \text{as } \lambda \to 0+ .
\end{align*}
Therefore, since this inequality holds for every $\varepsilon >0$,
\[
 \liminf_{\lambda\to 0+} \frac{(\Theta\circ \E ) (u+\lambda v) - (\Theta\circ\E ) (u)}{\lambda} \geq \langle \Theta' (\E (u)) \, f ,v\rangle .
\]
As a consequence $\Theta' (\E (u)) f \in \partial (\Theta \circ \E ) (u)$. 

If $\Theta' (\E (u)) \not= 0$, we may repeat the argument above with the inverse function $\Theta^{-1}$, which is continuous, strictly increasing, and differentiable at $(\Theta\circ\E ) (u)$, and we obtain the converse inclusion.   
\end{proof}

\begin{proof}[Proof of Theorem \ref{thm.main}]
Let $\varphi$ be as in the assumption, and let $U$ be a $\tau_\E$-neigh\-bour\-hood of $\varphi$ such that $\E$ satisfies the Kurdyka-{\L}ojasiewicz-Simon inequality in $U$. This means that there exists a strictly increasing $\Theta\in W^{1,1}_{loc} (\R )$ such that $| \partial (\Theta\circ \E ) (v) | \geq 1$ for every $v\in U$ with $0\not\in\partial\E (v)$. 

Since the energy is decreasing along the solution $u$, $\E (u(t)) \geq \E (\varphi )$ for every $t\in\R_+$. If $\E (u(t)) = \E (\varphi )$ for some $t\in\R_+$, then the energy is eventually constant along $u$, which implies that $u$ is eventually constant. In this case, there remains nothing to prove.   

Hence, we may assume that $\E (u(t)) >\E (\varphi )$ for every $t\in\R_+$. In this case, $\E (u)$ is strictly decreasing, and $\dot u(t) \not= 0$ for almost every $t\in\R_+$. By assumption and by Lemma \ref{lem.omega-limit} (c), there exists a sequence $(t_n)$ in $\R_+$ such that $\lim_{n\to\infty} t_n = \infty$ and $\lim_{n\to\infty} u(t_n) = \varphi$ in $\tau_\E$. Without loss of generality, we may assume that $u(t_n)\in U$ for every $n$. For every $n$ we set
\[
 s_n := \sup \{ s\in [t_n ,\infty ): u(t)\in U \text{ for every } t\in [t_n ,s]\} .
\]
Since $u$ is continuous with values in $(\dom{\E} , \tau_\E )$ and since $U$ is open in this space, $s_n >t_n$. For almost every $t\in [t_n , s_n )$, by the chain rule, the energy equality and Lemma \ref{lem.chain-rule},
\begin{align}
\nonumber -\frac{d}{dt} (\Theta \circ \E ) (u(t)) & = - \Theta' (\E (u(t))) \, \frac{d}{dt} \E (u(t)) \\
\nonumber  & = \frac12 \, \Theta' (\E (u(t))) \, (\| \dot u(t)\|_H^2 + |\partial\E (u(t))|^2 ) \\
\nonumber  & \geq \Theta' (\E (u(t))) \, \| \dot u (t)\|_H \, |\partial\E (u(t))| \\
\nonumber  & \geq \| \dot u (t)\|_H \, |\partial (\Theta \circ \E ) (u(t)) | \\
\label{est1} & \geq \|\dot u (t)\|_H . 
\end{align}
Integrating both sides, we obtain
\begin{align*}
 \| u(t) - u(t_n) \|_H & \leq \int_{t_n}^t \|\dot u (s) \|_H \; ds \\
 & \leq (\Theta \circ \E )(u(t_n)) - (\Theta \circ \E ) (u(t)) \\
 & \leq (\Theta \circ \E )(u(t_n)) - (\Theta \circ \E ) (\varphi ) . 
\end{align*}
Assume now that all $s_n$ are finite. Then, by continuity, the preceding inequality remains true for $t$ replaced by $s_n$, and thus
\begin{align*}
 \| u(s_n )-\varphi \|_H & \leq \| u(s_n ) - u(t_n) \|_H + \| u(t_n)-\varphi \|_H \\
 & \leq (\Theta\circ\E )(u(t_n)) - (\Theta\circ \E ) (\varphi ) + \| u(t_n)-\varphi \|_H .
\end{align*}
The convergence of $(u(t_n))$ to $\varphi$ in $\tau_\E$ and the continuity of $\Theta$ then imply that the right-hand side of this inequality converges to $0$ as $n\to\infty$. As a consequence,
\[
 \lim_{n\to\infty} u(s_n ) = \varphi \text{ in the norm topology of } H .
\]
This and Lemma \ref{lem.omega-limit} (a) yield
\[
 \lim_{n\to\infty} u(s_n ) = \varphi \text{ in the topology } \tau_\E ,
\]
which is, however, a contradiction since $u(s_n )\not\in U$ for every $n$. Hence, the assumption that all $s_n$ are finite was false. There thus exists $n$ such that $s_n = \infty$. In this case, the estimate \eqref{est1} implies $\dot u\in L^1 ([t_n , \infty );H )$, so that $u$ has finite length in $H$. By Cauchy's criterion, combined with Lemma \ref{lem.omega-limit} (a), we deduce $\lim_{t\to\infty} u(t) = \varphi$ in $\tau_\E$. 
\end{proof}

\begin{remark}
 We emphasize that the $\omega$-limit set of the solution $u$ in Theorem \ref{thm.main} is taken with respect to the norm topology in the ambient Hilbert space $H$. A condition for the nonemptiness of the $\omega$-limit is the condition that the range of $u$ is relatively compact in the norm topology of $H$. In many applications, this follows from mere boundedness of the solution in $H$, from the boundedness of the energy along $u$, and from standard compact embedding theorems. 

 Many articles on applications of the {\L}ojasiewicz-Simon inequality in the context of smooth gradient systems required in addition nonemptiness of the $\omega$-limit set in a finer topology. In the context of Example \ref{ex.2}, this would be the norm topology of the Sobolev space $H^1 (\Omega )$. This was usually verified by showing that the solution has relatively compact range in the underlying energy space, sometimes with considerable effort. Note that in Example \ref{ex.2}, the norm topology in $H^1 (\Omega )$ and the topology $\tau_\E$ coincide. Moreover, by Lemma \ref{lem.omega-limit} (c), the $\omega$-limit set with respect to the norm topology in $H$ and the $\omega$-limit set with respect to the topology $\tau_\E$ coincide.  
\end{remark}

\nocite{Ch03}
\nocite{Ku98}
\nocite{AmGiSa05}
\nocite{HaJe15}
\nocite{Hu06}
\nocite{Br73}

\bibliographystyle{plain}

%\bibliography{ralph}

\begin{thebibliography}{10}

\bibitem{AmGiSa05}
L.~Ambrosio, N.~Gigli, and G.~Savar\'e.
\newblock {\em Gradient {F}lows}.
\newblock Lectures in Mathematics ETH Z\"urich. Birkh\"auser, Basel, 2005.

\bibitem{Bar10}
V.~Barbu.
\newblock {\em Nonlinear differential equations of monotone types in {B}anach
  spaces}.
\newblock Springer Monographs in Mathematics. Springer, New York, 2010.

\bibitem{Br73}
H.~Brezis.
\newblock {\em Op\'erateurs maximaux monotones et semi-groupes de contractions
  dans les espaces de {H}ilbert}, volume~5 of {\em North Holland Mathematics
  Studies}.
\newblock North-Holland, Amsterdam, London, 1973.

\bibitem{Ch03}
R.~Chill.
\newblock On the {{\L}}ojasiewicz-{S}imon gradient inequality.
\newblock {\em J. Funct. Anal.}, 201:572--601, 2003.

\bibitem{HaJe15}
A.~Haraux and M.~A. Jendoubi.
\newblock {\em The convergence problem for dissipative autonomous systems}.
\newblock SpringerBriefs in Mathematics. Springer, Cham; BCAM Basque Center for
  Applied Mathematics, Bilbao, 2015.
\newblock Classical methods and recent advances, BCAM SpringerBriefs.

\bibitem{Hu06}
{S.-Z.} Huang.
\newblock {\em Gradient {I}nequalities: with {A}pplications to {A}symptotic
  {B}ehaviour and {S}tability of {G}radient-like {S}ystems}, volume 126 of {\em
  Mathematical Surveys and Monographs}.
\newblock Amer. Math. Soc., Providence, R.I., 2006.

\bibitem{Je98b}
M.~A. Jendoubi.
\newblock A simple unified approach to some convergence theorems of {L}.
  {S}imon.
\newblock {\em J. Funct. Anal.}, 153:187--202, 1998.

\bibitem{Ku98}
K.~Kurdyka.
\newblock On gradients of functions definable in {$o$}-minimal structures.
\newblock {\em Ann. Inst. Fourier (Grenoble)}, 48:769--783, 1998.

\bibitem{Lo63}
S.~{\L}ojasiewicz.
\newblock Une propri\'et\'e topologique des sous-ensembles analytiques r\'eels.
\newblock In {\em Colloques internationaux du C.N.R.S.: Les \'equations aux
  d\'eriv\'ees partielles, Paris (1962)}, pages 87--89. Editions du C.N.R.S.,
  Paris, 1963.

\bibitem{Lo65}
S.~{\L}ojasiewicz.
\newblock Ensembles semi-analytiques.
\newblock Preprint, I.H.E.S. Bures-sur-Yvette, 1965.

\bibitem{Si83}
L.~Simon.
\newblock Asymptotics for a class of non-linear evolution equations, with
  applications to geometric problems.
\newblock {\em Ann. of Math.}, 118:525--571, 1983.

\end{thebibliography}

\vskip 4 mm

\end{document}